\documentclass[letterpaper, 11 pt]{article}
\usepackage{fullpage, amsthm, amsmath, amssymb, amsfonts, color}
\usepackage{graphicx}

\newcommand{\newword}[1]{\textbf{\emph{#1}}}

\newcommand{\I}{\mathcal{I}}

\newtheorem{conj}{Conjecture}[section]

\newtheorem{theorem}[conj]{Theorem}

\newtheorem{proposition}[conj]{Proposition}

\newtheorem{lemma}[conj]{Lemma}

\newtheorem{corollary}[conj]{Corollary}

\theoremstyle{definition}
\newtheorem{definition}[conj]{Definition}

\newtheorem{remark}[conj]{Remark}

\begin{document}

\title{Rainbow Graphs and Switching Classes}

\author{Suho Oh, Hwanchul Yoo, and Taedong Yun}
\date{}
\maketitle

\begin{abstract}
A rainbow graph is a graph that admits a vertex-coloring such that every color appears exactly once in the neighborhood of each vertex. We investigate some properties of rainbow graphs. In particular, we show that there is a bijection between the isomorphism classes of $n$-rainbow graphs on $2n$ vertices and the switching classes of graphs on $n$ vertices.
\end{abstract}

\section{Introduction}

Throughout this paper, we will use the term \newword{graph} to refer to a graph on unlabeled vertices without loops or multiple edges, and \newword{labeled graph} to refer to a graph on labeled vertices without loops or multiple edges.
A vertex coloring (with $n$ colors) of a graph is called an \newword{$n$-rainbow coloring} if every color appears once, and only once, in each neighborhood of a vertex. Note that an $n$-rainbow coloring is not a proper coloring. A graph is called an \newword{$n$-rainbow} graph if the graph admits an $n$-rainbow coloring.

This term was initially coined by Woldar \cite{Woldar} but the same concept appeared earlier in Ustimenko \cite{Ustimenko} and in Lazebnik and Woldar \cite{Lazebnik} under the name of \emph{parallelotopic graph} and \emph{neighbor-complete coloring}. There are interesting instances of rainbow graphs appearing throughout mathematics and theoretical computer science literature. For example, one can naturally construct a rainbow graph from a group \cite{Woldar} or from a system of equations \cite{Lazebnik}, and the structure of such rainbow graphs has been shown to be closely related to the structure of the underlying algebraic structure.

We state here without proof some general facts about rainbow graphs, which are all easy to verify. Many of these facts can be found in \cite{Woldar}.

\begin{enumerate}
\item
An $n$-rainbow graph is $n$-regular.
\item
Each color occurs equally often in the vertex set of an $n$-rainbow graph.
\item
An $n$-rainbow graph contains a perfect matching, which is constructed by choosing edges whose endpoints have the same color.
\item
The number of vertices of an $n$-rainbow graph is a multiple of $2n$.
\end{enumerate}

On the other hand, \newword{switching} a vertex $v$ of a labeled graph $G$ means reversing the adjacency of $v$ and $w$ for every other vertex $w$: If $v$ and $w$ are connected then we delete the edge, otherwise we add an edge between them. This operation was defined by Seidel \cite{Seidel}. Two labeled graphs are \newword{switching equivalent} if one can be obtained from the other by a sequence of switching operations. Two unlabeled graphs are switching equivalent if we can label the vertices so that they become switching equivalent. The equivalent classes coming from this equivalence relation are called \newword{switching classes}. Mallows and Sloane \cite{Mallows} and Cameron \cite{Cameron} showed that the number of switching classes equals to the number of unlabeled Eulerian graphs. 

The main result of this paper is the following theorem.

\begin{theorem}
\label{thm:main}
The number of isomorphism classes of $n$-rainbow graphs with $2n$ vertices is the same as the number of switching classes of graphs with $n$ vertices.
\end{theorem}

\begin{corollary}
The number of isomorphism classes of $n$-rainbow graphs with $2n$ vertices is the same as the number of unlabeled Eulerian graphs (every vertex has even degree) with $n$ vertices.
\end{corollary}

We will explicitly construct a bijective map between the classes.

\section{Map from switching classes to $n$-rainbow graphs}

In this section, we introduce a map that sends a switching class of graphs having $n$ vertices to an isomorphism class of $n$-rainbow graphs having $2n$ vertices. This map will be shown to be a bijection in Section 3.

Choose a graph $G=(V,E)$ that has $n$ vertices without loops and multiple edges. Label the vertices from $1$ to $n$. Now let $\widetilde{G} = (\widetilde{V}, \widetilde{E})$ be a graph having $2n$ vertices such that:
\begin{enumerate}
\item the vertices are labeled by $\{1,\ldots,n,1',\ldots,n'\}$,
\item $(i,i') \in \widetilde{E}$,
\item $(i,j),(i',j') \in \widetilde{E}$ if $(i,j)\in E$,
\item $(i',j),(i,j') \in \widetilde{E}$ if $(i,j) \notin E$.
\end{enumerate}

Then the underlying unlabeled graph of $\widetilde{G}$ is $n$-rainbow because we can color the vertices labeled by $i$ and $i'$ with the color $i$, and this is clearly a rainbow coloring. An example is given in Figure~\ref{fig:map}.

\begin{figure}
	\centering
		\includegraphics[width=0.6\textwidth]{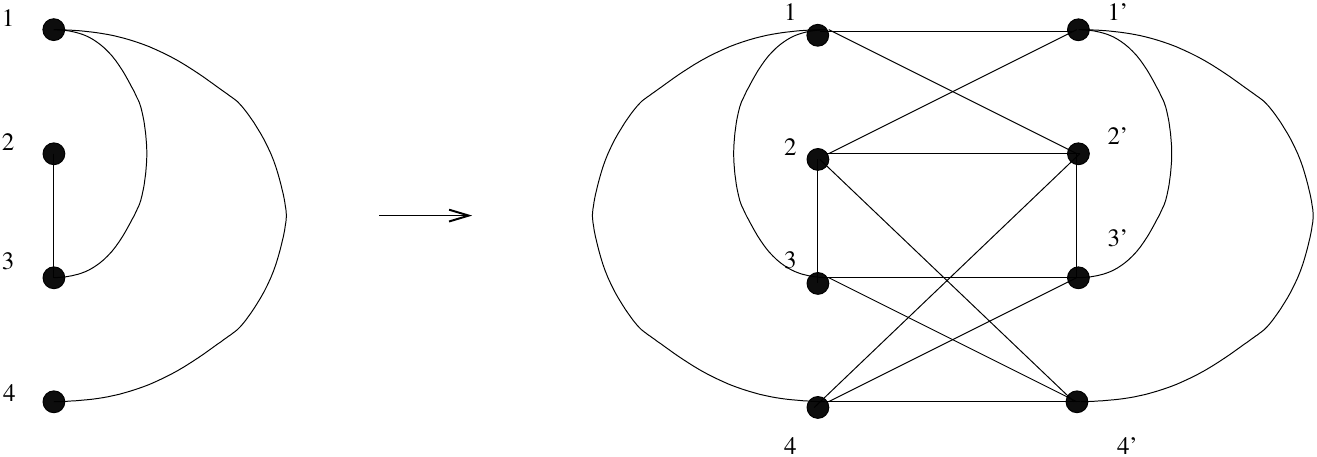}	
	\caption{The map that sends $G$ to $\widetilde{G}$.}
	\label{fig:map}
\end{figure}

We denote by $[G]$ the switching class of $G$. Now let us define the desired map $\psi$ by sending $[G]$ to the underlying unlabeled graph of $\widetilde{G}$. To see that $\psi$ is well-defined, let $H$ be the graph obtained from $G$ by switching at a vertex labeled by $i$. Let $\widetilde{H}$ be the graph with $2n$ vertices obtained from $H$ by the previous method. Now if we interchange the labels of $i,i'$ in $\widetilde{H}$, we recover $\widetilde{G}$. An example of this phenomenon is given in Figure~\ref{fig:welldef}.

\begin{figure}
	\centering
		\includegraphics[width=0.6\textwidth]{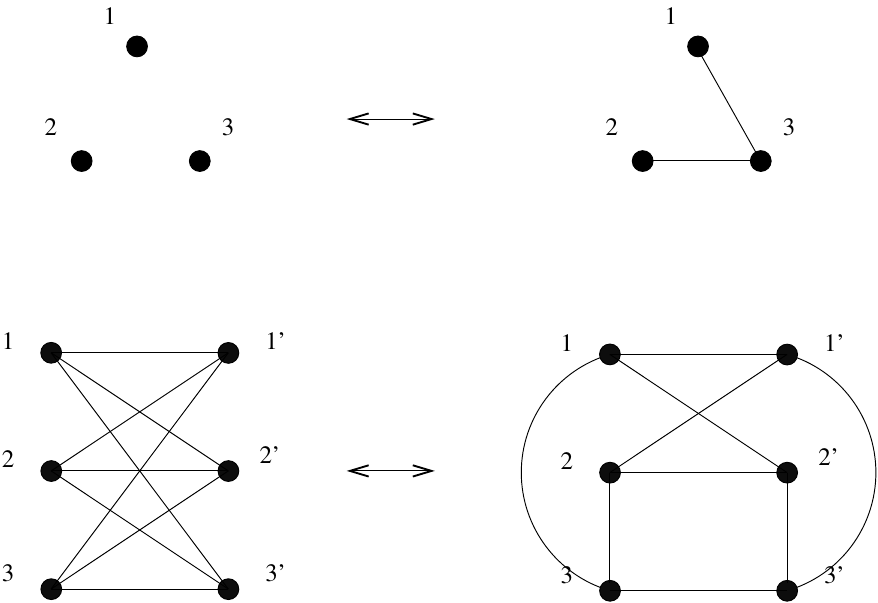}	
	\caption{The exchange equivalence.}
	\label{fig:welldef}
\end{figure}

\begin{proposition}
\label{prop:surj}
$\psi$ is surjective.
\end{proposition}
\begin{proof}
For any $n$-rainbow graph $\widetilde{G}$ on $2n$ vertices, assign an $n$-rainbow coloring on its vertex set. Now let $G$ be an induced subgraph of $\widetilde{G}$ having $n$ vertices on which each color appears once. Clearly, $\psi([G]) = \widetilde{G}$.
\end{proof}

\begin{remark}
A similar map was constructed in \cite{CC}, with a different coloring of graphs, to prove that the switching equivalence problem is polynomial time equivalent to the graph isomorphism problem.
\end{remark}

\section{The main proof}
In this section, we will prove our main result, Theorem~\ref{thm:main}. To do so, let us recall the definition of Seidel matrices, originally defined in \cite{Seidel}.

\begin{definition}
Let $G$ be a graph over $n$ vertices. The \newword{Seidel matrix} of $G$ is a symmetric $n\times n$ matrix with entries $a_{ij}$ such that:
\begin{itemize}
\item $a_{ii} = 0 $ for all $i \in [n]$,
\item $a_{ij} = a_{ji} = 1 $ if $(i,j) \in E(G)$,
\item $a_{ij} = a_{ji} = -1$ if $(i,j) \notin E(G)$.
\end{itemize}

\end{definition}

\bigskip

\begin{definition}
For a Seidel matrix $A$ we define $\widetilde{A}$ as
$$\widetilde{A} := \begin{pmatrix} A&I-A\\I-A&A \end{pmatrix}.$$
\end{definition}
Note that $\widetilde{A}$ is also a Seidel matrix. In particular, if $A$ is the Siedel matrix of $G$, then $\widetilde{A}$ is the Siedel matrix of $\tilde{G}$.

Let $A$ and $B$ be Seidel matrices of some graphs $G$ and $G'$, respectively. Then $G$ and $G'$ are isomorphic if and only if $B = PAP^{-1}$ for some permutation matrix $P$. Recall that the \emph{signed} permutation matrices are the matrices we get from permutation matrices by replacing $1$'s with $\pm 1$'s. Since switching at vertex $i$ of a graph corresponds to changing signs of entries in $i$-th row and $i$-th column of its Seidel matrix, $G$ and $G'$ are switching equivalent if and only if $B = QAQ^{-1}$ for some signed permutation matrix $Q$.

In order to prove Theorem~\ref{thm:main}, it only remains to prove that $\psi$ is injective, since surjectivity is given by Proposition~\ref{prop:surj}. We need to show that if $\tilde{G}$ is isomorphic to $\tilde{G'}$, then $G$ and $G'$ are in the same switching class. This can be restated in the following linear algebra language:

\begin{proposition}
\label{prop:3.3}
If two Seidel matrices $A$ and $B$ satisfy $\widetilde{A} = P\widetilde{B}P^{-1}$ where $P$ is a $2n\times 2n$ permutation matrix, then there exists an $n\times n$ signed permutation matrix $Q$ such that $A = QBQ^{-1}$.

\end{proposition}
We prove this proposition by a sequence of lemmas given below.
\begin{lemma}
\label{lem:3.4}
Let $P = \begin{pmatrix} P_1&P_2\\ P_3&P_4 \end{pmatrix}$ be a $2n\times 2n$ permutation matrix, where $P_1,\ldots, P_4$ are $n\times n$ matrices. The following equations hold:
\begin{enumerate}
\item[1)]
$P_1 P_3^T = 0$, $P_2 P_4^T = 0$, $P_3 P_1^T = 0$, $P_4 P_2^T = 0$.
\item[2)]
$P_1 P_1^T + P_2 P_2^T = I$, $P_3 P_3^T + P_4 P_4^T = I$.
\end{enumerate}
\end{lemma}

\begin{proof}
From $PP^T = I$ we have
$$\begin{pmatrix}P_1P_1^T + P_2P_2^T & P_1P_3^T + P_2 P_4^T \\ P_3 P_1^T + P_4 P_2^T & P_3P_3^T + P_4P_4^T \end{pmatrix} = \begin{pmatrix}I&0\\0&I\end{pmatrix}.$$
The claim follows from the fact that every entry of each of $P_i$'s is nonnegative.
\end{proof}

\begin{lemma}\label{lem:3.5}
In the assumption of Proposition \ref{prop:3.3} and Lemma \ref{lem:3.4}, let $Z$ denote $(P_1 - P_2 - P_3 + P_4)/2$. Then, we have $$A-I = Z (B-I) Z^T.$$
\end{lemma}

\begin{proof}
$\widetilde{A} = P\widetilde{B}P^{-1}$ implies
$$\begin{pmatrix} A & I-A \\ I-A & A \end{pmatrix} = \begin{pmatrix} P_1 & P_2 \\ P_3 & P_4 \end{pmatrix} \begin{pmatrix} B & I-B \\ I-B & B \end{pmatrix} \begin{pmatrix} P_1^T & P_3^T \\ P_2^T & P_4^T \end{pmatrix}.$$
It follows from Lemma \ref{lem:3.4} that:
$$A = XBX^T + I - XX^T,$$
$$A = XBY^T + I - XY^T,$$
$$A = YBY^T + I - YY^T,$$
$$A = YBX^T + I - YX^T,$$
where $X = P_1 - P_2$ and $Y = P_4 - P_3$. Summing these four equations, we get:
$$4A = (X+Y)B(X+Y)^T + 4I - (X+Y)(X+Y)^T.$$
Hence,
$$ A - I = \left( \frac{X+Y}{2} \right) (B-I) \left(\frac{X+Y}{2}\right)^T.$$
\end{proof}

\begin{definition}
Let $S$ be an $m\times m$ matrix. We call $S$ a \newword{signed half-permutation} matrix if every row and column of $S$ has either
\begin{enumerate}
\item[1)]
exactly one nonzero entry $\alpha$ with $\alpha = \pm 1$ , or
\item[2)]
exactly two nonzero entries $\alpha$ and $\beta$ with $\alpha = \pm \frac{1}{2}$, $\beta = \pm \frac{1}{2}.$
\end{enumerate}
\end{definition}

For instance, a matrix $$\begin{pmatrix} 0&-\frac{1}{2}&0&\frac{1}{2}\\ -\frac{1}{2}&-\frac{1}{2}&0&0\\ 0&0&-1&0\\ \frac{1}{2}&0&0&\frac{1}{2} \end{pmatrix}$$ is a signed half-permutation matrix.

\begin{definition}
\label{def:int}
Let $S$ be an $m\times m$ signed half-permutation matrix. We say that an $m\times m$ matrix $T$ is an \newword{integration} of $S$ if the following conditions hold for all $i = 1, \ldots, m$.
\begin{enumerate}
\item[1)]
If the $i$-th row(resp. column) of $S$ has only one nonzero entry(which is $\pm 1$), then the $i$-th row(resp. column) of $T$ is the same as the $i$-th row(resp. column) of $S$.
\item[2)]
If the $i$-th row(resp. column) of $S$ has two nonzero entries, so that it is of form\\ $[\cdots~ \frac{(-1)^p}{2}~ \cdots~ \frac{(-1)^q}{2}~ \cdots]$, then the $i$-th row(resp. column) of $T$ is obtained by doubling one of the nonzero entries and setting the other entry to zero. In other words, the $i$-th row(resp. column) of $T$ should look like either $[\cdots~ (-1)^p~\cdots~ 0~\cdots]$ or $[\cdots~ 0~\cdots~(-1)^q~\cdots]$.

\end{enumerate}
\end{definition}

Clearly an integration of a signed half-permutation matrix is a signed \emph{permutation} matrix. For instance, the following matrix $$\begin{pmatrix} 0&-1&0&0\\ -1&0&0&0\\ 0&0&-1&0\\ 0&0&0&1 \end{pmatrix}$$ is an integration of the following signed half-permutation matrix $$\begin{pmatrix} 0&-\frac{1}{2}&0&\frac{1}{2}\\ -\frac{1}{2}&-\frac{1}{2}&0&0\\ 0&0&-1&0\\ \frac{1}{2}&0&0&\frac{1}{2} \end{pmatrix}.$$

\begin{lemma}
Every signed half-permutation matrix has an integration.
\end{lemma}

\begin{proof}
We can find an integration $T$ of the signed half-permutation matrix $S$ by the following method:

First, pick an entry $\alpha_1$ of $S$ with value $\alpha_1 = \pm \frac{1}{2}$ in S. Replace $\alpha_1$ by $sgn(\alpha_1)$, where $sgn(\alpha)$ is $1$ if $\alpha$ is positive and is $-1$ if $\alpha$ is negative. Let $\alpha_2 = \pm \frac{1}{2}$ be the other nonzero entry in the same row as $\alpha_1$. Replace $\alpha_2$ by $0$. Let $\alpha_3 = \pm \frac{1}{2}$ be the other nonzero entry in the same column as $\alpha_2$. Replace $\alpha_2$ by $sgn(\alpha_2)$. By continuing this process, alternating row and column, we will end up at $\alpha_1$ again since there are only finite number of entries. If there is no entry with value $\pm \frac{1}{2}$ then we are done. Otherwise, we repeat this process until there is no entry with value $\pm \frac{1}{2}$.

Let $T$ be the matrix we get after performing this algorithm to $S$. Then, $T$ satisfies the two conditions of Definition \ref{def:int}.
\end{proof}

\begin{lemma}
\label{lem:ZtoQ}
Let $M$, $N$ be $m \times m$ matrices such that each entry $m_{ij}$ of $M$ satisfies $-1\leq m_{ij}\leq 1$ and every entry of $N$ is either $1$ or $-1$. If $MZ=N$ for some signed half-permutation matrix $Z$, then every entry of $M$ is either $1$ or $-1$. Furthermore, if $Q$ is an integration of $Z$ then we have $MQ = N$.
\end{lemma}
\begin{proof}
If $j$-th column of $Z$ has an entry $z_{ij}=\pm 1$, then $m_{ki}=\pm n_{kj}=\pm 1$ for all $k=1,\dots,m$. Otherwise, it has two nonzero entries, $z_{i_1j}=\pm \frac{1}{2}$ and $z_{i_2j}=\pm \frac{1}{2}$ for some $i_1$ and $i_2$. In this case, $-1\le m_{ki_1}z_{i_1j}+m_{ki_2}z_{i_2j}\le 1$, and the equalities are satisfied only if $m_{ki_1}=2n_{kj}z_{i_1}$ and $m_{ki_1}=2n_{kj}z_{i_1}$. Hence every entry of $M$ is either 1 or -1. Moreover, if $Q$ is an integration of $Z$, then it's a signed permutation matrix with $q_{ij}=z_{ij}$ if $z_{ij}=\pm 1$ and $q_{ij}=2z_{ij}$ or 0 if $z_{ij}=\pm \frac{1}{2}$. Therefore $m_{ki}q_{ij}=n_{kj}$, and we have $MQ=N$.
\end{proof}

Now we are ready to prove Proposition \ref{prop:3.3}, which would finish off the proof of Theorem \ref{thm:main}.

\begin{proof}[Proof of Proposition \ref{prop:3.3}]
By Lemma ~\ref{lem:3.5}, $A-I=Z(B-I)Z^T$, where $Z=(P_1-P_2-P_3+P_4)/2$. We can see that $Z$ is a signed half-permutation matrix. Let $Q$ be an integration of $Z$. Obviously $Q^T$ is an integration of $Z^T$. Since $A,B$ are Seidel matrices, every entry of $A-I$ is either $1$ or $-1$, and every entry of $Z(B-I)$ is between $-1$ and $1$. By Lemma~\ref{lem:ZtoQ}, every entry of $Z(B-I)$ is either $1$ or $-1$, and $A-I=Z(B-I)Q^T$. Using Lemma~\ref{lem:ZtoQ} again on the transpose of $Z(B-I)$, we get $Z(B-I)=Q(B-I)$. Therefore $A-I=Q(B-I)Q^T=QBQ^T-I$, hence $A=QBQ^T$.
\end{proof}


\textbf{Acknowledgment} We would like to thank Tanya Khovanova for introducing a puzzle that led us to the study of rainbow graphs: A sultan decides to check how wise his two wise men are. The sultan chooses a cell on a chessboard and shows it only to the first wise man. In addition, each cell on the chessboard either contains a pebble or is empty. Then, the first wise man either removes a pebble from a cell that contains a pebble or adds a pebble to an empty cell. Next, the second wise man must look at the board and guess which cell was chosen by the sultan. The two wise men are permitted to agree on the strategy beforehand. The strategy for ensuring that the second wise man will always guess the chosen cell comes from assigning a rainbow coloring on a $64$-regular graph.

We also thank the anonymous referees for pointing us to the reference \cite{Lazebnik}, \cite{Ustimenko}.

\bibliographystyle{plain}	
\bibliography{seidel}		

\emph{E-mail address}: suhooh@umich.edu

\emph{E-mail address}: hcyoo@kias.re.kr

\emph{E-mail address}: tedyun@math.mit.edu

\end{document}